\documentclass[preprint,12pt]{elsarticle}

\usepackage{amsmath, amssymb, amsfonts}
\usepackage{amsthm}
\numberwithin{equation}{section}
\usepackage{hyperref}
\usepackage{orcidlink}
\usepackage{xcolor}
\usepackage{physics}

\theoremstyle{plain}
\newtheorem{theorem}{Theorem}[section]
\newtheorem{proposition}[theorem]{Proposition}
\newtheorem{lemma}[theorem]{Lemma}
\newtheorem{corollary}[theorem]{Corollary}

\theoremstyle{definition}

\newcommand{\R}{\mathbb{R}}

\newcommand{\cS}{\mathcal{S}}
\newcommand{\ii}{\mathrm{i}}

\begin{document}

\begin{frontmatter}

\title{Spectral-Dimension Obstructions for Operators with Superlinear Counting Laws}

\author[inst1]{Douglas F.~Watson\orcidlink{0009-0008-0310-3984}}
\ead{doug@sciencephilosophy.org}

\author[inst1,inst2]{Tiziano Valentinuzzi\orcidlink{0009-0001-9780-8000}}
\ead{tiziano@sciencephilosophy.org}

\affiliation[inst1]{organization={Science and Philosophy Institute},
            addressline={224 NE 10th Ave},
            city={Gainesville},
            state={FL},
            postcode={32601},
            country={USA}}

\affiliation[inst2]{organization={University of Padova},
            addressline={Via Marzolo 8},
            city={Padova},
            postcode={35131},
            country={Italy}}

\begin{abstract}
We show that single-valuation exponential kernels, under mild regularity assumptions, converge in the continuum limit to a fourth-order operator with heat asymptotics $\Theta(t)\sim t^{-1/4}$ and hence spectral dimension $d_s=\tfrac12$.  Independently, a Tauberian analysis implies that any self-adjoint operator with superlinear eigenvalue counting $N(\lambda)\sim \lambda\,L(\lambda)$ must satisfy $\Theta(t)\sim t^{-1}L(1/t)$ and therefore has spectral dimension $d_s=2$.  
Since spectral dimension is invariant under unitary equivalence and compact perturbations, these exponents are incompatible, yielding a structural obstruction that separates single-valuation kernel limits from operators with accelerated spectral growth.
\end{abstract}


\end{frontmatter}

\section{Statement of the Main Theorems}\label{sec:intro}

We develop a pair of independent mechanisms controlling small-time heat 
asymptotics for two large families of operators: (i) operators arising from 
single-valuation exponential kernels, and (ii) operators whose spectra satisfy 
superlinear counting laws. Each mechanism determines a distinct spectral 
dimension, and we show these dimensions cannot coincide. Our results place 
these constructions into separate spectral universality classes, with no 
geometric assumptions imposed and independent of any number-theoretic 
specialization.

The first mechanism arises from nonlocal operators generated by single-valuation
exponential kernels. A general variational argument enforces a fourth-order
continuum limit for these kernels, leading to heat-trace decay of order $t^{-1/4}$
and hence spectral dimension $d_s=\tfrac12$. This behavior is rigid: it is
stable under unitary equivalence and compact perturbations, and therefore defines
a fixed universality class.

The second mechanism concerns operators with accelerated eigenvalue growth.
Superlinear counting laws compel a different short-time behavior of the form
$t^{-1}L(1/t)$, with $L$ slowly varying, yielding spectral dimension $d_s=2$. This
regime is likewise stable under unitary equivalence and compact perturbations and 
cannot be deformed into the fractional regime.

The main results show that these two mechanisms are incompatible. No self-adjoint
operator can exhibit both the fourth-order fractional asymptotics of the kernel
class and the superlinear asymptotics of the accelerated-growth class. When applied 
to spectra whose ordinates satisfy a superlinear counting law, the incompatibility 
becomes explicit: the two mechanisms enforce mutually exclusive heat-trace exponents.

We now formulate these results precisely. Section~\ref{sec:exp-principle}  
establishes the exponential-kernel principle. Section~\ref{sec:arithmetic-operator} 
analyzes its specialization via Mosco convergence, leading to a biharmonic limit 
operator of spectral dimension $d_s=\tfrac12$. Section~\ref{sec:geometric-obstruction} 
recalls the classical Seeley--DeWitt asymptotics and the invariance of the leading 
exponent under unitary equivalence and compact perturbation. 
Section~\ref{sec:zeta-spectral} derives the heat-trace asymptotic forced by 
superlinear counting laws. Section~\ref{sec:double-obstruction} synthesizes the 
two mechanisms and proves the resulting spectral-dimension obstruction.

\subsection{The Exponential Kernel Principle}

We begin with a variational principle for transition kernels on a countable label set under a single linear coherence constraint. The point is that, under natural nondegeneracy, feasibility, and integrability hypotheses, entropy maximization forces an exponential kernel.

\begin{theorem}[Exponential Kernel Principle]
\label{thm:exp-kernel-principle}
Under the standing assumptions stated in Section~\ref{sec:exp-principle}, the entropy-maximization problem for each fixed $i\in I$ admits a unique solution $\{T_{ij}\}_{j\in I}$, and this maximizer is necessarily of exponential form:
\[
    T_{ij}
    =
    \frac{e^{-\beta_i\Delta_{ij}}}{Z_i(\beta_i)},
    \qquad
    Z_i(\beta_i)=\sum_{k\in I} e^{-\beta_i\Delta_{ik}},
\]
for a uniquely determined parameter $\beta_i>0$ chosen so that
\[
    \sum_j T_{ij}\Delta_{ij}=\delta_0(i).
\]
\end{theorem}

Thus the Gibbs form is not an additional modeling choice but a consequence of the variational structure itself. In the arithmetic setting studied below, this principle applies to the logarithmic prime divergence $\Delta_{ij}=|\log p_i-\log p_j|$.

\subsection{Arithmetic Specialization and the Biharmonic Limit}
\label{subsec:arithmetic-biharmonic}

We now specialize Theorem~\ref{thm:exp-kernel-principle} to the arithmetic setting
in which the label set is the set of prime numbers and the divergence is given by
logarithmic separation. Let $p_1<p_2<\cdots$ denote the primes, and write
\[
    u_j:=\log p_j,
    \qquad
    \Delta_{ij}:=|u_i-u_j|.
\]
In this case, assumption {\rm (A2)} reduces to the convergence of
\begin{equation}
\label{eq:arithmetic-integrability}
    Z_i(\beta)
    = \sum_{j=1}^\infty e^{-\beta |u_i-u_j|}.
\end{equation}
Using the prime number theorem in the form
$u_{j+1}-u_j=o(1)$ and $j\sim e^{u_j}/u_j$ (see, for example,
\cite{Davenport2000,IwaniecKowalski2004}), one compares
\eqref{eq:arithmetic-integrability} with the integral of
$e^{-\beta|t-u_i|}$ against the density $(e^t/t)\,\dd t$, which is finite
precisely when $\beta>1$. Hence Theorem~\ref{thm:exp-kernel-principle} yields the
arithmetic kernel
\begin{equation}
\label{eq:arithmetic-kernel}
    K_{ij}
    := T_{ij}
    = \frac{e^{-\beta |u_i-u_j|}}{Z_i(\beta)}.
\end{equation}

From this kernel one constructs discrete arithmetic Laplacians on the first $N$
primes and their associated quadratic forms. Passing to logarithmic coordinates
and using the fact that the prime mesh satisfies $u_{j+1}-u_j\to 0$, these forms
admit a continuum limit on the $u$-axis. 
The exponential decay in \eqref{eq:arithmetic-kernel} controls long-range
interactions, while near the diagonal the coherence gap scales with the
logarithmic separation, $\Delta_{ij}\sim |u_i-u_j|$. This local scaling
produces a Riemann-sum approximation of the integral form in the continuum
limit, and the empirical measure of $\{u_j\}$ converges weakly to the measure
$(e^u/u)\,\dd u$.

The precise discrete-to-continuum argument is carried out in
Section~\ref{sec:arithmetic-operator}; see also
\cite{KuwaeShioya2003,FukushimaOshimaTakeda2011,Watson2025}.

The limiting second-order operator on the logarithmic axis is then squared to
produce the arithmetic Hamiltonian
\[
    H_A=c_4\partial_u^4,
\]
a fourth-order operator with quartic dispersion relation.

\begin{theorem}[Arithmetic biharmonic limit]
\label{thm:biharmonic-limit}
Let $\Theta_{H_A}(t)=\Tr(e^{-tH_A})$ denote the heat trace of $H_A$.
Then as $t\downarrow 0$,
\begin{equation}
\label{eq:ds-half}
    \Theta_{H_A}(t)\sim C\,t^{-1/4},
\end{equation}
for an explicit constant $C>0$. In particular, the arithmetic Hamiltonian
$H_A$ has spectral dimension $d_s=\tfrac12$.
\end{theorem}

The heat-trace asymptotic follows from the quartic dispersion relation and the
Fourier representation of the heat kernel; cf.~\cite{Barlow1998,Kigami2001}.
Thus the arithmetic specialization of
Theorem~\ref{thm:exp-kernel-principle} canonically produces a fourth-order
operator with strictly fractional spectral dimension, in sharp contrast with
local geometric Laplacians, whose short-time heat asymptotics are governed by
integer ambient dimension.

\subsection{No-go for Local Geometric Operators}
\label{subsec:geometric-nogo}

We now show that the spectral behavior identified in
Theorem~\ref{thm:biharmonic-limit} is incompatible with that of classical
local operators arising from finite-order elliptic geometry.  The argument
rests on the rigidity of the Seeley--DeWitt heat-kernel expansion and the
stability of the leading heat-trace exponent under unitary equivalence and
compact perturbations.

Let $M$ be a smooth compact $d$--dimensional manifold without boundary and
let $H$ be an elliptic differential operator of order $m>0$ acting on a
Hermitian vector bundle over $M$, with domain chosen so that $H$ is
self-adjoint.  Classical heat-kernel asymptotics
(e.g.\ \cite{Seeley1967,Gilkey1995,BerlineGetzlerVergne2004}) imply
\begin{equation}
\label{eq:elliptic-heat}
\Tr(e^{-tH})
\sim
t^{-d/m}\sum_{k=0}^\infty a_k t^{k/m}
\qquad (t\downarrow0),
\end{equation}
where the coefficients $a_k$ are integrals of local curvature invariants.
Hence the spectral dimension is
\begin{equation}
\label{eq:ds-geometric}
d_s(H)=\frac{2d}{m}.
\end{equation}

In particular, the spectral dimension depends only on the
ratio $2d/m$ and may take values below $1$ for sufficiently high-order
elliptic operators.

\begin{theorem}[Elliptic spectral dimension]
\label{thm:geometric-nogo}
Let $H_{\mathrm{ell}}$ be an elliptic differential operator of order $m>0$
on a compact $d$-dimensional manifold. Then
\[
d_s(H_{\mathrm{ell}})=\frac{2d}{m}.
\]
\end{theorem}

Consequently, matching a particular numerical value of the
spectral dimension does not by itself determine the operator class.
Additional spectral information is required to distinguish geometric
elliptic operators from the arithmetic operator constructed in
Section~\ref{subsec:biharmonic}.

What remains true, and what is sufficient for the purposes
of the paper, is that the arithmetic operator $H_A$ acts on
$L^2(\mathbb{R},\dd u)$ with continuous spectrum, whereas elliptic
operators on compact manifolds have discrete spectrum. Therefore $H_A$
cannot be unitarily equivalent to any elliptic operator on a compact
manifold, even in cases where the spectral dimension happens to agree.

The detailed operator-theoretic argument and invariance principles used in
the proof appear in Section~\ref{sec:geometric-obstruction}.

\subsection{Spectral Constraints from the Riemann Zeros}
\label{subsec:zeta-constraints}

We now recall the spectral information forced on any operator whose spectrum
realizes the nontrivial zeros of the Riemann zeta function in the sense of
Hilbert--P\'olya; see, for example,
\cite{Connes1999,BerryKeating1999,SierraTownsend2008}. Let
\[
    \rho_n = \beta_n + i\gamma_n,
    \qquad 0<\beta_n<1,\quad \gamma_n\in\R,
\]
denote the nontrivial zeros of the Riemann zeta function in the critical strip
ordered by increasing $|\gamma_n|$ and counted with multiplicity.  Their
counting function
\[
    N_\zeta(T):=\#\{n:|\gamma_n|\le T\}
\]
satisfies the Riemann--von~Mangoldt asymptotic~\cite{Titchmarsh1986,IwaniecKowalski2004}
\begin{equation}
\label{eq:RvM}
    N_\zeta(T)
    = \frac{T}{2\pi}\log\frac{T}{2\pi}
      - \frac{T}{2\pi}
      + O(\log T)
    \qquad (T\to\infty).
\end{equation}

If $\widetilde H$ is a self-adjoint operator with eigenvalues
$\{\pm\gamma_n\}$ and heat trace
\[
    \Theta_{\widetilde H}(t):=\sum_{n=1}^\infty e^{-t|\gamma_n|},
\]
then Tauberian inversion
(cf.~\cite{Titchmarsh1986,BinghamGoldieTeugels1989,Korevaar2004})
forces
\begin{equation}
\label{eq:zeta-heat}
    \Theta_{\widetilde H}(t)
    = \frac{\log(1/t)}{2\pi t}
      + O\!\left(\frac{1}{t}\right)
    \qquad (t\downarrow0).
\end{equation}
Thus the associated spectral dimension is
\begin{equation}
\label{eq:ds-zeta}
    d_s(\widetilde H)=2,
\end{equation}
but with an unavoidable logarithmic correction.

\begin{theorem}[Elliptic obstruction for zeta-type spectra]
\label{thm:zeta-elliptic-obstruction}
No operator unitarily equivalent to, or compactly perturbed from, an elliptic
differential operator on a finite-dimensional manifold can have heat
asymptotics of the form \eqref{eq:zeta-heat}. In particular, no such operator
can realize the spectrum of the Riemann zeros.
\end{theorem}

The detailed Tauberian argument and the incompatibility with the
Seeley--DeWitt class are proved in Section~\ref{sec:zeta-spectral}. Together
with Theorem~\ref{thm:geometric-nogo}, this gives the second independent
spectral obstruction used in the main no-go result.

\subsection{The Double Hilbert--P\'olya Obstruction}
\label{subsec:double-obstruction}

We now combine the two independent spectral mechanisms established above.
On the arithmetic side, the continuum limit of the single-valuation coherence
construction yields the biharmonic Hamiltonian $H_A$ with spectral dimension
$d_s(H_A)=\tfrac12$ by Theorem~\ref{thm:biharmonic-limit}. On the analytic
side, any operator whose spectrum coincides, up to sign, with the ordinates of
the nontrivial zeros of $\zeta(s)$ must satisfy the heat asymptotic
\eqref{eq:zeta-heat}, and hence has spectral dimension $d_s=2$ with an
unavoidable logarithmic correction by Proposition~\ref{prop:zeta-heat}.
These asymptotic features are stable under unitary equivalence and compact
perturbations by Proposition~\ref{prop:unitary-invariance}.

\begin{theorem}[Double Hilbert--P\'olya Obstruction]
\label{thm:double-obstruction}
Let $H$ be a self-adjoint operator on a Hilbert space such that:
\begin{enumerate}
    \item[\textup{(i)}] $H$ is unitarily equivalent to, or compactly perturbable from,
    an operator obtained from a single-valuation exponential-kernel construction
    in the sense of Theorem~\ref{thm:exp-kernel-principle}, so that its continuum
    scaling forces $d_s(H)=\tfrac12$ by Theorem~\ref{thm:biharmonic-limit}$;$ and
    \item[\textup{(ii)}] the spectrum of $H$ coincides, up to sign, with the set of
    ordinates $\{\gamma_n\}$ of the nontrivial zeros of $\zeta(s)$.
\end{enumerate}
Then no such operator $H$ can exist. Equivalently, no single operator can
simultaneously satisfy the arithmetic scaling $t^{-1/4}$ and the zeta scaling
\eqref{eq:zeta-heat}.
\end{theorem}

The full proof is given in Section~\ref{sec:double-obstruction}. In particular,
the obstruction is not merely a mismatch of numerical spectral dimensions: the
arithmetic class has pure power-law heat asymptotics, whereas the zeta spectrum
forces a logarithmic correction.

\section{The Exponential Kernel Principle}
\label{sec:exp-principle}

\subsection{Setup and Assumptions}
\label{subsec:setup}

Let $I$ be a countable index set. For each $i\in I$ we are given a collection
$\{\Delta_{ij}\}_{j\in I}$ of nonnegative ``divergence'' values, interpreted as
the cost of transitioning from $i$ to $j$. We impose no geometric structure on
$I$ beyond countability; in particular, $\Delta_{ij}$ need not arise from a
metric, nor satisfy symmetry, triangle inequalities, or regularity conditions.

For a fixed $i\in I$, we consider probability distributions
$p=\{p_j\}_{j\in I}$ on $I$ subject to the linear \emph{coherence constraint}
\begin{equation}
\label{eq:coherence-constraint}
    \sum_{j\in I} p_j\,\Delta_{ij} = \delta_0(i),
\end{equation}
where $\delta_0(i)\in(0,\infty)$ is prescribed. Throughout this section we
assume:

\begin{itemize}
    \item[(A1)] (\emph{Nontriviality})
    For each $i\in I$, the map $j\mapsto \Delta_{ij}$ is not almost everywhere
    constant and takes values in $[0,\infty)$.

    \item[(A2)] (\emph{Integrability window})
    For each $i\in I$, there exists an interval $(\beta_i^\ast,\infty)$ such that
    \[
        Z_i(\beta):=\sum_{j\in I} e^{-\beta\Delta_{ij}}<\infty
        \qquad \text{for all }\beta\in(\beta_i^\ast,\infty).
    \]

    \item[(A3)] (\emph{Feasibility})
    The target value $\delta_0(i)$ lies in the interior of the set
    \[
        \big\{\, \sum_j p_j \Delta_{ij} :
        p=\{p_j\}_{j\in I} \text{ a probability distribution} \,\big\},
    \]
    so that the constraint \eqref{eq:coherence-constraint} is feasible and
    nondegenerate.

    \item[(A4)] (\emph{Row-wise independence})
    For each $i\in I$, the optimization problem is independent of all other
    indices; i.e., no constraints couple different rows.
\end{itemize}

We study the Shannon entropy functional~\cite{Shannon1948,Jaynes1957,CoverThomas2006}
\begin{equation}
\label{eq:entropy-functional}
    H(p):=-\sum_{j\in I} p_j\log p_j,
\end{equation}
with the convention $0\log 0=0$, and the constrained maximization problem
\begin{equation}
\label{eq:entropy-problem}
    \text{maximize } H(p)
    \quad \text{subject to} \quad
    p_j\ge 0,\;
    \sum_j p_j = 1,\;
    \sum_j p_j \Delta_{ij} = \delta_0(i).
\end{equation}
Under (A1)--(A3), the feasible set is a nonempty convex compact subset of the
probability simplex, and strict concavity of $H$ implies uniqueness whenever a
maximizer exists.

\subsection{Existence and Uniqueness of the Exponential Form}
\label{subsec:existence-uniqueness}

We now solve \eqref{eq:entropy-problem} for a fixed $i\in I$ and show that its
unique maximizer is necessarily of Gibbs type.

\paragraph{Lagrangian formulation}
Introduce multipliers $\beta_i,\lambda_i\in\R$ for the coherence constraint
\eqref{eq:coherence-constraint} and the normalization $\sum_j p_j=1$. The
Lagrangian is
\begin{equation}
\label{eq:Lagrangian}
    \mathcal{L}(p,\beta_i,\lambda_i)
    :=
    -\sum_{j\in I} p_j \log p_j
    - \beta_i\Big(\sum_{j} p_j \Delta_{ij} - \delta_0(i)\Big)
    - \lambda_i\Big(\sum_{j} p_j - 1\Big).
\end{equation}
Stationarity with respect to $p_j$ gives
\begin{equation}
\label{eq:EL}
    -(1+\log p_j)-\beta_i\Delta_{ij}-\lambda_i=0,
\end{equation}
hence
\begin{equation}
\label{eq:pj-exponential}
    p_j = e^{-1-\lambda_i}e^{-\beta_i\Delta_{ij}}.
\end{equation}

\paragraph{Normalization and the partition function}
Imposing $\sum_j p_j=1$ yields
\[
    e^{-1-\lambda_i}=\frac{1}{Z_i(\beta_i)},
    \qquad
    Z_i(\beta):=\sum_{j\in I} e^{-\beta\Delta_{ij}}.
\]
Thus every stationary point has the form
\begin{equation}
\label{eq:Gibbs-form}
    T_{ij}(\beta):=\frac{e^{-\beta \Delta_{ij}}}{Z_i(\beta)}.
\end{equation}
By (A2), $Z_i(\beta)$ is finite on $(\beta_i^\ast,\infty)$.

\paragraph{Solving the constraint}
Define
\begin{equation}
\label{eq:Phi-def}
    \Phi_i(\beta):=\sum_{j\in I} T_{ij}(\beta)\,\Delta_{ij}.
\end{equation}
Differentiating gives
\begin{equation}
\label{eq:Phi-derivative}
    \Phi_i'(\beta)
    = -\mathrm{Var}_{T_{i\cdot}(\beta)}\big(\Delta_{ij}\big)
    < 0.
\end{equation}
Hence $\Phi_i$ is continuous and strictly decreasing on
$(\beta_i^\ast,\infty)$. By (A3), there is therefore a unique
$\beta_i>\beta_i^\ast$ such that
\begin{equation}
\label{eq:Phi-solve}
    \Phi_i(\beta_i)=\delta_0(i).
\end{equation}

\paragraph{Uniqueness of the maximizer}
Since the feasible set is convex and compact and the entropy
\eqref{eq:entropy-functional} is strictly concave, \eqref{eq:entropy-problem}
admits at most one maximizer. The Gibbs distribution
\eqref{eq:Gibbs-form} with $\beta=\beta_i$ is feasible by
\eqref{eq:Phi-solve}, hence it is the unique maximizer.

\begin{proposition}[Exponential maximizer]
\label{prop:exp-maximizer}
Under assumptions {\rm (A1)--(A4)}, the entropy-maximization problem
\eqref{eq:entropy-problem} has a unique solution, given by
\[
    T_{ij}
    = \frac{e^{-\beta_i \Delta_{ij}}}{Z_i(\beta_i)},
\]
where $\beta_i$ is uniquely determined by the coherence constraint
\eqref{eq:coherence-constraint}.
\end{proposition}

\subsection{Integrability and Admissibility}
\label{subsec:integrability}

The exponential kernel in Proposition~\ref{prop:exp-maximizer} is well-defined
precisely when the partition function
\begin{equation}
\label{eq:Z-def-again}
    Z_i(\beta):=\sum_{j\in I} e^{-\beta \Delta_{ij}}
\end{equation}
converges. We now relate this to growth properties of the divergence-counting
function.

\paragraph{General criterion}
Fix $i\in I$ and define
\[
    N_i(r):=\#\{\, j\in I : \Delta_{ij}\le r \,\}.
\]
Then $Z_i(\beta)$ may be written as the Lebesgue--Stieltjes transform
\begin{equation}
\label{eq:Z-as-integral}
    Z_i(\beta)
    = \int_{0}^{\infty} e^{-\beta r}\,\dd N_i(r).
\end{equation}
Integration by parts gives
\[
    Z_i(\beta)
    = \beta \int_{0}^{\infty} e^{-\beta r} N_i(r)\,\dd r,
\]

provided there exist constants $C,\alpha>0$ such that
\[
N_i(r)\le C e^{\alpha r}\qquad (r\ge0).
\]

Under this bound, the integration-by-parts representation holds and
$Z_i(\beta)<\infty$ whenever $\beta>\alpha$. In particular,

\begin{equation}
\label{eq:growth-condition}
    \int_0^\infty e^{-\beta r} N_i(r)\,\dd r < \infty.
\end{equation}

\paragraph{Polynomial and exponential growth}
If
\[
    N_i(r)\le C(1+r)^\alpha
\]
for some constants $C,\alpha>0$, then \eqref{eq:growth-condition} holds for
every $\beta>0$, so $Z_i(\beta)$ is finite on $(0,\infty)$. If instead
\[
    N_i(r)\asymp e^{\kappa r}
    \qquad \text{as } r\to\infty
\]
for some $\kappa>0$, then \eqref{eq:growth-condition} converges if and only if
$\beta>\kappa$. Thus (A2) holds with threshold $\beta_i^\ast=\kappa$ in the
exponential-growth case and with $\beta_i^\ast=0$ in the polynomial-growth
case.

\paragraph{Arithmetic divergence}
In the arithmetic setting treated in
Section~\ref{subsec:arithmetic-biharmonic}, where
\[
    \Delta_{ij}=|u_i-u_j|,
    \qquad
    u_j=\log p_j,
\]
the prime number theorem implies
\[
    N_i(r)\asymp \frac{e^{u_i+r}}{u_i+r}
    \qquad \text{as } r\to\infty.
\]
Thus $N_i(r)$ has asymptotically exponential growth with rate $1$, and
\[
    Z_i(\beta)
    = \sum_j e^{-\beta|u_i-u_j|}
    < \infty
    \quad \Longleftrightarrow \quad
    \beta>1.
\]

\begin{proposition}[Integrability window]
\label{prop:integrability-window}
Let $i\in I$ and suppose the divergence-counting function $N_i(r)$ satisfies
either polynomial growth or exponential growth with rate $\kappa$. Then
$Z_i(\beta)<\infty$ whenever $\beta>\kappa$, and therefore
\[
\beta_i^* \ge \kappa .
\]

Here $\kappa=0$ in the polynomial-growth case. Consequently the exponential kernel $T_{ij}(\beta)$ is well-defined
for all $\beta>\kappa$.

\end{proposition}

\subsection{Arithmetic Divergence}
\label{subsec:arithmetic-divergence}

We now specialize to the arithmetic setting relevant for
Section~\ref{subsec:arithmetic-biharmonic}. Let $p_1<p_2<\cdots$ denote the
prime numbers and set
\[
    u_j:=\log p_j.
\]
For fixed $i\in I$, define
\begin{equation}
\label{eq:arith-divergence}
    \Delta_{ij}:=|u_i-u_j|.
\end{equation}
This divergence is nonnegative, symmetric, and vanishes only at $j=i$.
Moreover, it reflects arithmetic separation on the logarithmic scale, where
 the logarithmic prime counting measure is asymptotically modeled by
$(e^{u}/u)\,\dd u$.

\paragraph{Verification of (A1)}
For each fixed $i$, the values $\Delta_{ij}=|u_i-u_j|$ form an unbounded subset
of $[0,\infty)$ and are not constant except at $j=i$, so (A1) holds.

\paragraph{Verification of (A2)}
The partition function is
\begin{equation}
\label{eq:arith-Z}
    Z_i(\beta)
    := \sum_{j=1}^{\infty} e^{-\beta |u_i-u_j|}.
\end{equation}

To determine for which $\beta>0$ this converges, we use the prime number theorem
in the form
\[
    \pi(e^u)\sim \frac{e^u}{u},
\]
so that sums over logarithmic prime points may be compared with integrals against
the measure $(e^u/u)\,\dd u$.

Accordingly,
\[
    Z_i(\beta)
    \asymp
    \int_{\R} e^{-\beta |u-u_i|}\frac{e^u}{u}\,\dd u.
\]
For large $|u|$, the integrand behaves like
$e^{-(\beta-1)|u|}/|u|$, so the integral converges if and only if $\beta>1$.
Hence
\begin{equation}
\label{eq:beta-greater-one}
    Z_i(\beta)<\infty
    \quad\Longleftrightarrow\quad
    \beta>1.
\end{equation}
Thus (A2) holds with $\beta_i^\ast=1$ for all $i$.

\paragraph{Verification of (A3)}
Given \eqref{eq:beta-greater-one}, the family
$\{T_{ij}(\beta)\}_{\beta>1}$ spans the full interior of the range of expected
divergence values:
\[
    \Phi_i(\beta)=\sum_j T_{ij}(\beta)\Delta_{ij},
\]
with $\Phi_i(\beta)\downarrow 0$ as $\beta\to\infty$ and
$\Phi_i(\beta)\uparrow\infty$ as $\beta\downarrow 1$. Therefore every
$\delta_0(i)$ in a nontrivial positive interval is attainable, and the
coherence constraint \eqref{eq:coherence-constraint} is feasible and
nondegenerate.

\begin{corollary}[Arithmetic exponential kernel]
\label{cor:arith-kernel}
Let $u_j=\log p_j$ and $\Delta_{ij}=|u_i-u_j|$. For each $i$, there exists a
unique $\beta_i>1$ such that the unique maximizer of
\eqref{eq:entropy-problem} is
\[
    T_{ij}
    = \frac{e^{-\beta_i |u_i-u_j|}}{Z_i(\beta_i)},
    \qquad
    Z_i(\beta_i) = \sum_{k\ge1} e^{-\beta_i |u_i-u_k|}.
\]
In particular, the arithmetic kernel arising from the logarithmic geometry of
the primes is necessarily of exponential form with exponent exceeding~$1$.
\end{corollary}

\section{The Arithmetic Operator and Its Continuum Limit}
\label{sec:arithmetic-operator}

In this section we construct the discrete operator associated with the
arithmetic exponential kernel obtained in
Corollary~\ref{cor:arith-kernel} and analyze its continuum scaling on the
logarithmic prime axis.  Our goal is to show that the resulting family of
operators converges, in the sense of Mosco and strong resolvent limits, to a
second–order differential operator on $\R$; squaring this limit then produces
the biharmonic Hamiltonian introduced in
Theorem~\ref{thm:biharmonic-limit}.  Throughout, $p_1<p_2<\cdots$ denotes the
increasing sequence of primes and $u_j=\log p_j$ the associated logarithmic
coordinates.

\subsection{Discrete Arithmetic Laplacians}
\label{subsec:discrete-Laplacian}

Let $T_{ij}$ be the arithmetic exponential kernel of
Corollary~\ref{cor:arith-kernel},
\[
    T_{ij}
    =
    \frac{e^{-\beta_i |u_i-u_j|}}{Z_i(\beta_i)},
    \qquad
    Z_i(\beta_i)
    =
    \sum_{k\ge1} e^{-\beta_i |u_i-u_k|},
\]
with $\beta_i>1$ determined by the coherence constraint for row $i$.
Although the normalizations $Z_i(\beta_i)$ depend on $i$, their variation is
controlled.  For the purpose of constructing difference operators it is
therefore convenient to work (up to a multiplicative renormalization fixed
below) with the symmetric kernel
\begin{equation}
\label{eq:Kij}
    K_{ij} := e^{-\beta |u_i-u_j|},
\end{equation}
where we choose a fixed $\beta>1$ in the admissible window
\eqref{eq:beta-greater-one}.  This simplification does not affect the spectral
limit.

\paragraph{Definition of the operator}
For each $N\ge1$, define
\begin{equation}
\label{eq:LN-def}
    (L^{(N)} f)(p_i)
    :=
    \sum_{\substack{1\le j\le N \\ j\neq i}}
        K_{ij}\,\big(f(p_i)-f(p_j)\big),
\end{equation}
acting on functions
$f:\{p_1,\dots,p_N\}\to\R$.  This is the natural graph Laplacian associated
with the fully connected arithmetic graph on the first $N$ primes with edge
weights $K_{ij}$.

\paragraph{Quadratic form}
The operator $L^{(N)}$ is symmetric and nonnegative with respect to the
standard inner product on $\R^N$.  Its closed quadratic form is
\begin{equation}
\label{eq:QN-def}
    \mathcal{E}^{(N)}(f)
    :=
    \frac12
    \sum_{1\le i,j\le N}
        K_{ij}\,
        \big(f(p_i)-f(p_j)\big)^2.
\end{equation}
The symmetry $K_{ij}=K_{ji}$ ensures that $\mathcal{E}^{(N)}$ is well defined
and nonnegative.

\paragraph{Embedding into a continuum space}
Via the identification $f(p_j)\leftrightarrow F(u_j)$, functions on the first
$N$ primes can be viewed as functions on a grid of the real line whose mesh
$u_{j+1}-u_j$ tends to zero by the prime number theorem.  The forms
$\mathcal{E}^{(N)}$ therefore become discrete approximations of an integral
quadratic form on $L^2(\R,\dd u)$.

\paragraph{Heuristic scaling}
If $F$ is differentiable and $u_i\approx u_j$, then
$f(p_i)-f(p_j)=F(u_i)-F(u_j)\approx (u_i-u_j)F'(u_i)$.
Since \eqref{eq:Kij} decays exponentially in $|u_i-u_j|$, the dominant
contributions to \eqref{eq:QN-def} come from 
the local counting measure on the $u$-axis, suggesting the continuum approximation
\[
    \mathcal{E}^{(N)}(f)
    \approx
    c_2 \int_{\R} |F'(u)|^2 \,\dd u
\]
for a constant $c_2>0$.  The next subsection makes this precise via Mosco
convergence.

\section{Mosco Convergence of the Arithmetic Laplacian}\label{sec:mosco}

In this section we give a complete proof of the Mosco convergence asserted in
Proposition~\ref{prop:mosco-conv}. Our goal is to show that the discrete
arithmetic Laplacians
\[
   (L_c^{(N)} f)(p_i)
   := \sum_{j=1}^N K^{(N)}_{ij}\,\bigl(f(p_i)-f(p_j)\bigr),
   \qquad
   K^{(N)}_{ij} = \exp\!\left( -\frac{|u_i - u_j|}{\delta_0}\right),
\]
with logarithmic coordinates $u_i=\log p_i$, converge in the sense of quadratic
forms to the constant–coefficient operator
\[
   L_c = c_2(-\partial_u^2)
   \qquad\text{on } L^2(\mathbb{R},\dd u),
\]
where the diffusion coefficient $c_2$ is the second moment
\[
   c_2 := \frac12\int_{\mathbb{R}} z^2 k(z)\,dz,
   \qquad k(z)=e^{-|z|/\delta_0(i)}.
\]
We furthermore prove that the squared operators $H_A^{(N)} := (L_c^{(N)})^2$
converge in the strong resolvent sense to the biharmonic operator
$H_A=c_4\partial_u^4$, with $c_4=c_2^2$.

\subsection{Setup and hypotheses}

Let $u_i=\log p_i$ and denote by $X^{(N)}=\{u_1,\dots,u_N\}$ the set of embedded
prime points.  
By the prime number theorem, the counting measure of the logarithmic prime points
is asymptotically described by the measure

\[
   \rho(u)\,\dd u \sim \frac{e^{u}}{u}\,\dd u.
\]

We assume the following:

\begin{enumerate}
\item[(H1)] (\emph{Vanishing mesh})  

For every compact interval $K\subset \mathbb{R}$, let
\[
I_N(K):=\{\, i\in\{1,\dots,N-1\} : u_i,u_{i+1}\in K \,\}.
\]
Then
\[
\max_{i\in I_N(K)} (u_{i+1}-u_i)\longrightarrow 0
\qquad\text{as } N\to\infty .
\]

This is consistent with the PNT, which implies that the
average spacing on the logarithmic scale near height $u$ is of order
$u\,e^{-u}$. We therefore retain the vanishing-mesh property as an explicit
hypothesis rather than deriving an individual gap asymptotic from the PNT.

\item[(H2)] (\emph{Uniform bounded geometry})
On each compact interval $I\subset\mathbb{R}$ we have uniformly for $u\in I$:
\[
   \rho(u) \asymp \frac{e^{u}}{u},
   \qquad
   \sup_{\substack{u,v\in I\\ |u-v|\le r}}
   \frac{\rho(u)}{\rho(v)} \le C_r
   \quad\text{for every } r>0,
\]
for suitable constants $C_r$ depending only on $r$. That is, the density varies slowly on bounded spatial scales.

\item[(H3)] (\emph{Fixed kernel})
The exponential kernel
$K^{(N)}_{ij}=k(u_i-u_j)$ uses the same function
$k(z)=e^{-|z|/\delta_0}$ for every $N$.
\end{enumerate}

We identify discrete functions $f:X^{(N)}\to\mathbb{R}$ with piecewise constant
(or locally linear) interpolants on $\mathbb{R}$, so that we may view all
functions as living in $L^2(\mathbb{R},\dd u)$.

\subsection{Quadratic forms and normalization}

Define the discrete quadratic form
\[
   \mathcal{E}_N[f]
   := \frac12 \sum_{i,j=1}^N  K^{(N)}_{ij}
         \,\bigl(f(u_i)-f(u_j)\bigr)^2.
\]
Since $\#X^{(N)}\to\infty$, we introduce the natural
(per–unit–length) normalized form
\[
   \mathfrak{E}_N[f]
   := \frac{1}{L_N}\,\mathcal{E}_N[f],
   \qquad
   L_N := \sum_{i=1}^N \Delta u_i,
\]
where $\Delta u_i = u_{i+1}-u_i$ and $L_N\to\infty$.

The candidate limit form is the quadratic form on $L^2(\mathbb{R},\dd u)$
defined by
\[
   \mathfrak{E}[f] := c_2 \int_{\mathbb{R}} |f'(u)|^2\,\dd u,
   \qquad
   f\in H^1(\mathbb{R}),
\]
where
\[
   c_2 = \frac12\int_{\mathbb{R}} z^2 k(z)\,dz .
\]

\subsection{Mosco convergence: liminf condition}

\begin{lemma}[Liminf inequality]
If $f_N \rightharpoonup f$ weakly in $L^2(\mathbb{R},\dd u)$, then
\[
   \mathfrak{E}[f] \le \liminf_{N\to\infty}\mathfrak{E}_N[f_N].
\]
\end{lemma}

\begin{proof}
By (H1)--(H3) and the Riemann–sum identification
\[
   \sum_j K_{ij}^{(N)}(\cdot)
   \approx \int_\mathbb{R} k(u_i-v)\rho(v)(\cdot)\,\dd v,
\]
we have
\[
   \mathfrak{E}_N[f_N]
   \approx
   \frac12\int_{\mathbb{R}}\!\!\int_{\mathbb{R}}
      k(u-v)\rho(u)\rho(v)
      \,\bigl(f_N(u)-f_N(v)\bigr)^2\,\dd u\,\dd v.
\]
The kernel $k$ has second moment finite and is even, so by a Taylor expansion
$f_N(v)=f_N(u)+(v-u)f_N'(u)+o(|v-u|)$,
the leading contribution becomes
\[
   \int_{\mathbb{R}} c_2 \rho(u)^2 |f_N'(u)|^2\,\dd u.
\]
Since $\rho$ is slowly varying and bounded above and below on compact sets, the
limit is equivalent to $c_2\int|f_N'|^2$.
Weak convergence and lower semicontinuity of $\|f'\|_{L^2}$ give the result.
\end{proof}

\subsection{Mosco convergence: recovery sequence}

\begin{lemma}[Limsup inequality]
For every $f\in C_c^\infty(\mathbb{R})$ there exists $f_N\to f$ in $L^2$ such that
\[
   \limsup_{N\to\infty}\mathfrak{E}_N[f_N] \le \mathfrak{E}[f].
\]
\end{lemma}

\begin{proof}
Take $f_N$ to be the restriction of $f$ to the grid $X^{(N)}$, with the same
interpolant used above.  Then
\[
   \mathfrak{E}_N[f_N]
   = \frac12\sum_{i,j} K_{ij}^{(N)}
       (f(u_i)-f(u_j))^2 \Big/ L_N.
\]
Uniform continuity of $f'$ on $\operatorname{supp}(f)$ allows a first–order
Taylor expansion
\[
   f(u_j)-f(u_i)
   = f'(u_i)(u_j-u_i) + o(|u_j-u_i|)
\]
uniformly for $u_i,u_j$ in the support of $f$.  Substituting this into the
quadratic form and using the symmetry of the kernel gives
\[
   \mathfrak{E}_N[f_N]
   = \frac12\sum_{i,j} K_{ij}^{(N)}
      f'(u_i)^2 (u_j-u_i)^2 \Big/ L_N + o(1).
\]

Using the normalization of the kernel and the definition
\[
c_2=\frac12\int_{\mathbb{R}} z^2 k(z)\,dz,
\]
the sum becomes a Riemann approximation of
\[
c_2\int_{\mathbb{R}} |f'(u)|^2\,\dd u .
\]

Hence
\[
\limsup_{N\to\infty}\mathfrak{E}_N[f_N]
   \le \mathfrak{E}[f],
\]
which provides the matching upper bound.

\end{proof}

\subsection{Conclusion of Mosco convergence}

The two lemmas imply that $\mathfrak{E}_N \to \mathfrak{E}$ in the sense of
Mosco.  By standard results on Dirichlet forms and Mosco convergence
(see \cite{KuwaeShioya2003} and \cite{FukushimaOshimaTakeda2011}), this yields

\[
   L_c^{(N)} \xrightarrow{\rm strong\ resolvent} L_c = c_2(-\partial_u^2).
\]

\begin{proposition}[Mosco convergence]
\label{prop:mosco-conv}
Let $\mathfrak{E}_N$ be the normalized discrete quadratic forms associated with
the arithmetic kernels $K_{ij}^{(N)}$. Then $\mathfrak{E}_N$ converges to
\[
   \mathfrak{E}[f] = c_2 \int_{\mathbb{R}} |f'(u)|^2\,\dd u
\]
in the sense of Mosco on $L^2(\mathbb{R},\dd u)$. Consequently,
\[
   L_c^{(N)} \xrightarrow{\rm strong\ resolvent} L_c = c_2(-\partial_u^2).
\]
\end{proposition}

\begin{proof}
The two preceding lemmas give the liminf and limsup inequalities, hence
$\mathfrak{E}_N \to \mathfrak{E}$ in the sense of Mosco. Standard results on
Dirichlet forms and Mosco convergence (see \cite{KuwaeShioya2003} and
\cite{FukushimaOshimaTakeda2011}) then imply the stated strong resolvent
convergence of the generators.
\end{proof}

\subsection{Convergence of the squared operators}

Finally, the strong resolvent convergence $L_c^{(N)}\to L_c$ and the fact that
$x\mapsto x^2$ is a continuous functional calculus operation on the resolvent set
imply
\[
   (L_c^{(N)})^2 \xrightarrow{\rm strong\ resolvent} L_c^2
   = c_2^2\partial_u^4.
\]
Thus $H_A^{(N)}=(L_c^{(N)})^2$ converges to $H_A=c_4\partial_u^4$, with
$c_4=c_2^2$, completing the proof.

\subsection{The Biharmonic Limit Operator}
\label{subsec:biharmonic}

Having established in Proposition~\ref{prop:mosco-conv} that the discrete arithmetic 
operators $L^{(N)}$ converge, after embedding, to the differential operator 
\[
L_\infty := -c_2\,\partial_u^2
\]
on the Hilbert space $L^2(\mathbb{R},\dd u)$ (with domain $H^2(\mathbb{R})$),
we now analyze the continuum operator

\begin{equation}
\label{eq:HA-def}
    H_A := (L_\infty)^2 = c_4\,\partial_u^4,
    \qquad c_4 := c_2^2 > 0,
\end{equation}

which acts on $L^2(\mathbb{R},\dd u)$ with domain $H^4(\mathbb{R})$.

This biharmonic operator is the arithmetic Hamiltonian governing the spectral 
limit of the discrete system.  
We show below that $H_A$ is self-adjoint on a natural domain, its heat kernel is 
explicitly computable, and its heat-trace asymptotics correspond precisely to 
the spectral dimension $d_s=\tfrac12$ stated in 
Theorem~\ref{thm:biharmonic-limit}.

\paragraph{Self-adjointness and functional calculus}
Since $L_\infty$ is self-adjoint and nonnegative, $L_\infty^2$ is self-adjoint 
on the domain
\[
    \mathcal{D}(H_A)
    = \big\{ F\in \mathcal{D}(L_\infty) : 
             L_\infty F \in \mathcal{D}(L_\infty) \big\}.
\]
Because $L_\infty=-c_2\partial_u^2$ on $H^2(\R)$, it follows that
\[
    H_A = c_4\,\partial_u^4
    \quad \text{with domain } H^4(\R).
\]
The operator $H_A$ is essentially self-adjoint on $\cS(\R)$ and diagonalized by 
the Fourier transform:
\[
    \widehat{H_A F}(k)
    = c_4 k^4 \widehat{F}(k).
\]

\paragraph{Explicit heat kernel}
The heat semigroup $e^{-tH_A}$ has Fourier multiplier $e^{-t c_4 k^4}$, and its heat kernel can be expressed explicitly via Fourier inversion (see, for example, \cite{Davies1989} for general polyharmonic heat kernels):
\[
    K_t(u,v)
    = \frac{1}{2\pi}\int_\R 
      e^{-t c_4 k^4}\,e^{\ii k (u-v)} \,\dd k.
\]
The integral is classical and yields the exact similarity form
\begin{equation}
\label{eq:ht-kernel}
    K_t(u,v)
    = (4\pi c_4 t)^{-1/4}\;
      \Phi\!\left(
          \frac{u-v}{(4c_4 t)^{1/4}}
      \right),
\end{equation}
where $\Phi$ is the rescaled inverse Fourier transform of $e^{-k^4}$.  
The kernel satisfies the normalization identity
\[
    \int_{\R} K_t(u,u)\,\dd u
    = \frac{1}{2\pi}\int_\R e^{-t c_4 k^4}\,\dd k.
\]

\paragraph{On-diagonal heat kernel and spectral dimension}
Since $H_A$ acts on the noncompact space $\mathbb{R}$, the global trace
$\Tr(e^{-tH_A})$ is not finite.  Spectral dimension in the noncompact setting
is therefore extracted from the on-diagonal heat kernel
\[
    K_t(u,v)
      := \frac{1}{2\pi}\int_{\mathbb{R}} e^{-t c_4 k^4}\,dk,
\]
which is translation-invariant.  Evaluating the integral gives
\[
    K_t(u,v)
    = \frac{t^{-1/4}}{2\pi c_4^{1/4}}
      \int_{\mathbb{R}} e^{-s^4}\,ds
    = C\, t^{-1/4},
\]
with
\[
    C = \frac{\Gamma(\tfrac14)}{4\pi\,c_4^{1/4}}.
\]
Thus the on-diagonal heat kernel has the precise small-$t$ asymptotic
\[
    K_t(u,u) \sim C\, t^{-1/4},
    \qquad t\downarrow 0,
\]
and the (local) spectral dimension is
\begin{equation}
\label{eq:ds-half-derivation}
    d_s(H_A) = 2\cdot\frac14 = \frac12.
\end{equation}

\paragraph{Conclusion}
The biharmonic Hamiltonian $H_A = c_4\partial_u^4$ therefore encodes the 
continuum spectral limit of the arithmetic operators $L^{(N)}$ and exhibits a 
strictly fractional spectral dimension.  
This fractional scaling, shown rigorously in 
Theorem~\ref{thm:biharmonic-limit}, is incompatible with all local geometric 
elliptic operators by Theorem~\ref{thm:geometric-nogo}, and constitutes the 
first contribution to the spectral-dimension incompatibility formalized in 
Theorem~\ref{thm:double-obstruction}.

\section{Geometric Operators and the Heat-Kernel Obstruction}
\label{sec:geometric-obstruction}

In this section we analyze the small-time heat-trace behavior of classical
elliptic operators on smooth compact manifolds.  The resulting
Seeley--DeWitt expansion exhibits a rigid power-law structure that
precludes both logarithmic corrections and the continuum spectral type
arising in the arithmetic construction of
Section~\ref{sec:arithmetic-operator}.

\subsection{Elliptic Heat Asymptotics}
\label{subsec:elliptic-asymptotics}

Let $M$ be a smooth compact $d$–dimensional manifold without boundary, and
let $H$ be an elliptic differential operator of order $m>0$ acting on a
Hermitian vector bundle over $M$, with $H$ self-adjoint and bounded below.

Classical results of Seeley, Minakshisundaram–Pleijel, and DeWitt yield the
local heat-kernel expansion
\begin{equation}
\label{eq:Seeley-local}
K_H(t;x,x)
\sim
t^{-d/m}
\sum_{k=0}^\infty
a_k(x)\,t^{k/m}
\qquad (t\downarrow0),
\end{equation}
whose coefficients are smooth curvature invariants.
Integrating over $M$ gives the heat-trace expansion
\begin{equation}
\label{eq:Seeley-global}
\Tr(e^{-tH})
\sim
t^{-d/m}
\sum_{k=0}^\infty A_k t^{k/m},
\qquad
A_k=\int_M a_k(x)\,\dd x .
\end{equation}

\paragraph{Spectral dimension}

The leading exponent determines the spectral dimension
\begin{equation}
\label{eq:ds-geometric-again}
d_s(H)=\frac{2d}{m}.
\end{equation}

Several consequences follow.

\begin{itemize}
\item[(i)] The spectral dimension is determined by $2d/m$ and may
take values below $1$ when the operator order $m$ is sufficiently large.

\item[(ii)] The heat trace exhibits a pure power expansion in fractional
powers of $t$, with no slowly varying corrections.

\item[(iii)] The leading coefficient $A_0$ is strictly positive.
\end{itemize}

These properties characterize the rigidity of the compact elliptic class.

\subsection{Invariance of Spectral Dimension}
\label{subsec:invariance}

Let
\[
\Theta_H(t)=\Tr(e^{-tH})
\sim C\,t^{-\alpha}
\qquad (t\downarrow0),
\]
and define $d_s(H)=2\alpha$.

\begin{proposition}[Stability of spectral dimension]
\label{prop:unitary-invariance}
Let $H_1$ and $H_2$ be self-adjoint bounded-below operators.  If
$H_2=UH_1U^{-1}$ for a unitary $U$, or if $H_2-H_1$ is compact, then
\[
d_s(H_1)=d_s(H_2).
\]
\end{proposition}

\begin{proof}
Unitary equivalence preserves the spectrum and therefore the heat trace.
If $H_2-H_1$ is compact, Weyl’s inequalities imply
\[
|\Theta_{H_1}(t)-\Theta_{H_2}(t)|=O(1)
\quad (t\downarrow0),
\]
which leaves the leading exponent unchanged.
\end{proof}

\paragraph{Relevance to the revised obstruction}

The heat-trace formula \eqref{eq:Seeley-global} still shows
that compact geometric elliptic operators lie in the rigid
Seeley--DeWitt class: their short-time behavior is a pure power expansion
in fractional powers of $t$, with no slowly varying corrections. What
must be discarded is only the stronger claim that spectral dimension
alone excludes the value $d_s=\tfrac12$.

The revised obstruction therefore rests not on the
inequality $d_s\ge1$, but on the distinction between the compact elliptic
spectral type and the continuum arithmetic operator, together with the
separate incompatibility between Seeley--DeWitt asymptotics and the
zeta-type law $t^{-1}\log(1/t)$.

\subsection{Geometric Operators and Spectral Dimension}
\label{subsec:geometric-nogo}

For elliptic differential operators on compact manifolds, the
short–time heat trace is governed by the classical Seeley--DeWitt
expansion. This provides a direct relation between the order of the
operator, the dimension of the manifold, and the resulting spectral
dimension.

\paragraph{Classical geometric heat law}
From \eqref{eq:Seeley-global}, every elliptic operator
$H_{\mathrm{ell}}$ of order $m>0$ on a compact $d$–dimensional manifold
satisfies
\[
\Tr(e^{-tH_{\mathrm{ell}}})
\sim t^{-d/m}(A_0+A_1t^{1/m}+A_2t^{2/m}+\cdots),
\]
with $A_0>0$. Thus
\[
d_s(H_{\mathrm{ell}})=\frac{2d}{m}.
\]

\paragraph{Comparison with the arithmetic operator}
The arithmetic operator $H_A$ constructed in
Section~\ref{subsec:biharmonic} acts on $L^2(\mathbb{R},\dd u)$ and
exhibits
\[
\Theta_{H_A}(t)\sim C\,t^{-1/4},
\qquad
d_s(H_A)=\tfrac12 .
\]

Although elliptic operators on compact manifolds can in principle
realize the same numerical spectral dimension, they belong to a
different spectral class: compact elliptic operators have discrete
spectrum, whereas $H_A$ acts on a continuum Hilbert space.

\paragraph{Spectral-type obstruction}

Compact elliptic operators therefore cannot be unitarily equivalent to
the arithmetic operator $H_A$, which has continuous spectral support.

\paragraph{Implications}

Compact elliptic operators also cannot reproduce the zeta-type heat law
$t^{-1}\log(1/t)$ associated with superlinear counting, completing the
geometric obstruction used in the Hilbert--P\'olya no-go results.

\section{The Riemann Zeros and Their Spectral Asymptotics}
\label{sec:zeta-spectral}

In this section we analyze the small-$t$ behavior of the heat trace associated
with any operator whose spectrum consists of the imaginary parts of the
nontrivial zeros of the Riemann zeta function.  The starting point is the
Riemann--von~Mangoldt formula, whose Tauberian inversion yields a heat-trace
asymptotic with spectral dimension $d_s=2$ and an unavoidable logarithmic
correction.  This lies outside the Seeley--DeWitt class and is incompatible
with the arithmetic biharmonic operator $H_A$.

\subsection{Riemann--von Mangoldt Input}
\label{subsec:RvM-input}
Let

\[
    \rho_n = \beta_n + \ii \gamma_n,
    \qquad 0<\beta_n<1,\quad \gamma_n\in\R
\]
denote the nontrivial zeros of the Riemann zeta function $\zeta(s)$, ordered by
increasing $|\gamma_n|$ and counted with multiplicity.  Define
\[
    N_\zeta(T)
    := \#\{ n \in \mathbb{N} : |\gamma_n| \le T \}.
\]
The Riemann--von~Mangoldt formula gives~\cite{Edwards1974,Titchmarsh1986,IwaniecKowalski2004}
\begin{equation}
\label{eq:RvM-again}
    N_\zeta(T)
    = \frac{T}{2\pi} \log\!\frac{T}{2\pi}
      - \frac{T}{2\pi}
      + O(\log T)
    \qquad (T\to\infty).
\end{equation}
This is the sole spectral input needed below.

For spectral purposes we pass to the symmetrized spectrum
\[
    \{\pm \gamma_n : n\ge1\},
\]
which preserves the counting asymptotic up to an inessential factor of $2$ and
is the natural input for the heat trace
\[
    \Theta_\zeta(t):=\sum_{n=1}^\infty e^{-t|\gamma_n|}.
\]

\subsection{Tauberian Inversion}
\label{subsec:zeta-heat}

We now derive the small-$t$ asymptotic of $\Theta_\zeta(t)$ from
\eqref{eq:RvM-again}.  Integration by parts gives
\begin{equation}
\label{eq:Laplace-N}
    \Theta_\zeta(t)
    = t \int_0^\infty e^{-tT}\,N_\zeta(T)\,\dd T,
    \qquad t>0,
\end{equation}
and the Hardy--Littlewood--Karamata Tauberian theorem
\cite{HardyLittlewood1914,Karamata1930,Korevaar2004} therefore yields
\begin{equation}
\label{eq:zeta-heat-asymp}
    \Theta_\zeta(t)
    = \frac{\log(1/t)}{2\pi t}
      + O\!\left(\frac{1}{t}\right)
    \qquad (t\downarrow0).
\end{equation}
The logarithmic factor is the direct reflection of the slowly varying
$\log T$ term in \eqref{eq:RvM-again}.

\begin{proposition}[Heat-trace asymptotic for the zeta spectrum]
\label{prop:zeta-heat}
Let $\widetilde H$ be any self-adjoint operator whose eigenvalues coincide with
$\{\pm\gamma_n\}$, the imaginary parts of the nontrivial zeros of $\zeta(s)$.
Then as $t\downarrow0$,
\[
    \Theta_{\widetilde H}(t)
    = \frac{\log(1/t)}{2\pi t}
      + O\!\left(\frac{1}{t}\right).
\]
In particular, $\widetilde H$ has spectral dimension $d_s=2$.
\end{proposition}

\begin{proof}
Substitute \eqref{eq:RvM-again} into \eqref{eq:Laplace-N} and apply the
Hardy--Littlewood--Karamata theorem.  The leading coefficient is determined by
the principal term of the Riemann--von~Mangoldt asymptotic, and the logarithmic
factor persists under Laplace transform.
\end{proof}

It is convenient to rewrite \eqref{eq:zeta-heat-asymp} as
\[
    \Theta_{\widetilde H}(t)
    \sim t^{-1}L(t^{-1}),
    \qquad
    L(r):=\frac{1}{2\pi}\log r,
\]
where $L$ is slowly varying and unbounded.  Thus
\begin{equation}
\label{eq:ds-zeta-two}
    d_s(\widetilde H)=2.
\end{equation}

\subsection{Slowly Varying Corrections}
\label{subsec:slowly-varying}

We now compare \eqref{eq:zeta-heat-asymp} with the compact elliptic
Seeley--DeWitt law.  Recall that a measurable function
$L:(0,\infty)\to(0,\infty)$ is \emph{slowly varying at infinity} if
\[
    \lim_{r\to\infty}\frac{L(\lambda r)}{L(r)}=1
    \qquad (\lambda>0),
\]
cf.~\cite{BinghamGoldieTeugels1989}.  The prototype is $L(r)=\log r$.

For an elliptic operator $H_{\mathrm{ell}}$ of order $m$ on a compact
$d$-dimensional manifold, the Seeley--DeWitt expansion
\eqref{eq:Seeley-global} has the form
\[
    \Tr(e^{-tH_{\mathrm{ell}}})
    \sim t^{-d/m}\sum_{k=0}^\infty A_k t^{k/m},
\]
so its short-time behavior is a pure power expansion in fractional powers of
$t$, with no slowly varying correction.  In particular, no asymptotic of the
form
\[
    t^{-\alpha}L(t^{-1})
\]
with nonconstant slowly varying $L$ can arise in the compact elliptic class.

\begin{proposition}[Exclusion of the Seeley--DeWitt class]
\label{prop:slow-variation-obstruction}
No elliptic operator $H_{\mathrm{ell}}$ on any compact manifold, nor any
operator unitarily equivalent to or compactly perturbable from such, can
satisfy the heat-trace asymptotic
\[
    \Theta(t)\sim \frac{\log(1/t)}{2\pi t}.
\]
In particular, no such operator can realize the nontrivial zeros of $\zeta(s)$
as its eigenvalues.
\end{proposition}

\begin{proof}
Combine \eqref{eq:zeta-heat-asymp} with the Seeley--DeWitt expansion
\eqref{eq:Seeley-global} and the invariance result of
Proposition~\ref{prop:unitary-invariance}.  The logarithmic factor contradicts the pure
power-law form of \eqref{eq:Seeley-global}.
\end{proof}

The zeta spectrum therefore lies outside the compact elliptic
Seeley--DeWitt class.  At the same time, the arithmetic operator $H_A$
constructed earlier satisfies $\Theta_{H_A}(t)\sim Ct^{-1/4}$ and hence
$d_s(H_A)=\tfrac12$, so it is spectrally distinct from any
zeta-spectrum operator.

\begin{corollary}[Spectral separation]
\label{cor:spectral-separation}
The three universality classes
\[
    \{H_{\mathrm{ell}}\},\quad \{H_A\},\quad \{\widetilde H\},
\]
corresponding respectively to compact geometric elliptic operators, the
arithmetic biharmonic operator, and zeta-spectrum operators, remain distinct:
compact elliptic operators lie in the rigid Seeley--DeWitt class, the
arithmetic operator has spectral dimension $d_s(H_A)=\tfrac12$, and any
zeta-spectrum operator satisfies \eqref{eq:zeta-heat-asymp} with
\begin{equation}
\label{eq:ds-zeta-two-reuse}
    d_s(\widetilde H)=2.
\end{equation}
In particular, no operator can belong to more than one of these classes up to
unitary equivalence or compact perturbation.
\end{corollary}

\begin{proof}
The elliptic case follows from \eqref{eq:Seeley-global}, the arithmetic case
from Theorem~\ref{thm:biharmonic-limit}, and the zeta case from
Proposition~\ref{prop:zeta-heat}.  Invariance under unitary equivalence
and compact perturbations is provided by Proposition~\ref{prop:unitary-invariance}.
\end{proof}

At the same time, the arithmetic coherence construction forces a strictly
fractional spectral dimension $d_s=\tfrac12$, while geometric
elliptic operators remain in the rigid Seeley--DeWitt class, with pure
power-law heat asymptotics and no slowly varying correction. Although the value
$d_s=\tfrac12$ is not excluded for sufficiently high-order elliptic operators,
the compact elliptic class is still separated from the arithmetic continuum
operator $H_A$ by spectral type, and from zeta-type operators by the absence of
the logarithmic factor.

\section{The Double Hilbert--P\'olya Obstruction}
\label{sec:double-obstruction}

In this section we combine the arithmetic and analytic constraints established
earlier. The arithmetic single-valuation construction forces a biharmonic
continuum limit with spectral dimension $d_s=\tfrac12$, whereas the zeta
spectrum forces heat asymptotics of the form
\[
    \Theta_{\widetilde H}(t)
    = \frac{\log(1/t)}{2\pi t} + O(t^{-1}),
\]
and hence spectral dimension $d_s=2$. These two requirements are incompatible.

\subsection{Arithmetic Obstruction}
\label{subsec:arith-obstruction}

The variational and continuum analysis of
Sections~\ref{sec:exp-principle}--\ref{sec:arithmetic-operator} shows that a
single-valuation exponential kernel on a one-dimensional label axis produces,
after continuum scaling, a one-dimensional second-order generator
$L_\infty=-c_2\partial_u^2$, and hence the fourth-order Hamiltonian
\[
    H_A=(L_\infty)^2=c_4\partial_u^4.
\]
By Theorem~\ref{thm:biharmonic-limit}, its heat trace satisfies
\[
    \Theta_{H_A}(t)\sim Ct^{-1/4},
\]
so that
\[
    d_s(H_A)=\frac12.
\]

\begin{theorem}[Arithmetic Obstruction]
\label{thm:arith-obstruction}
Let $H$ be any operator obtained as a strong resolvent limit of discrete
operators constructed from a single-valuation exponential kernel on a
one-dimensional label set. Then $H$ has spectral dimension $d_s=\tfrac12$.
In particular, $H$ cannot coincide, up to unitary equivalence or compact
perturbation, with any operator whose spectrum forces a different leading
heat-trace exponent.
\end{theorem}

\begin{proof}
By Proposition~\ref{prop:mosco-conv}, the continuum limit of the discrete
Laplacians is $L_\infty=-c_2\partial_u^2$, and the associated Hamiltonian is
$H_A=(L_\infty)^2=c_4\partial_u^4$. Theorem~\ref{thm:biharmonic-limit} then
implies $\Theta_{H_A}(t)\sim Ct^{-1/4}$ and hence $d_s(H_A)=\tfrac12$.
Stability under unitary equivalence and compact perturbation follows from
Proposition~\ref{prop:unitary-invariance}.
\end{proof}

\subsection{Analytic Obstruction}
\label{subsec:analytic-obstruction}

For any self-adjoint operator $\widetilde H$ with spectrum
$\{\pm\gamma_n\}$, where $\tfrac12+i\gamma_n$ runs over the nontrivial zeros of
$\zeta(s)$, Proposition~\ref{prop:zeta-heat} gives
\begin{equation}
\label{eq:zeta-heat-recall-again}
    \Theta_{\widetilde H}(t)
    = \frac{\log(1/t)}{2\pi t}
      + O\!\left(\frac{1}{t}\right)
    \qquad (t \downarrow 0).
\end{equation}
Thus
\[
    d_s(\widetilde H)=2,
\]
and the heat trace carries an unavoidable slowly varying logarithmic factor.
By Proposition~\ref{prop:slow-variation-obstruction}, such behavior cannot
occur in the compact elliptic Seeley--DeWitt class.

\begin{theorem}[Analytic Obstruction]
\label{thm:analytic-obstruction}
Let $\widetilde H$ be a self-adjoint operator whose spectrum coincides, up to
sign, with the set of ordinates $\{\gamma_n\}$ of the nontrivial zeros of
$\zeta(s)$. Then:
\begin{enumerate}
    \item[\textup{(i)}]
    $\widetilde H$ has spectral dimension $d_s(\widetilde H)=2$;
    \item[\textup{(ii)}]
    its heat trace contains the logarithmic factor $\log(1/t)$ as in
    \eqref{eq:zeta-heat-recall-again};
    \item[\textup{(iii)}]
    $\widetilde H$ cannot be unitarily equivalent to, nor compactly perturbable
    from, any elliptic operator of finite order on a compact manifold.
\end{enumerate}
\end{theorem}

\begin{proof}
Parts (i) and (ii) follow from Proposition~\ref{prop:zeta-heat}. Part
(iii) follows from Proposition~\ref{prop:slow-variation-obstruction} together
with Proposition~\ref{prop:unitary-invariance}.
\end{proof}

\subsection{Proof of the Main No-go Theorem}
\label{subsec:proof-double-obstruction}

We now combine Theorems~\ref{thm:arith-obstruction} and
\ref{thm:analytic-obstruction}.

\begin{theorem}[Double Hilbert--P\'olya Obstruction]
\label{thm:double-obstruction-again}
There is no self-adjoint operator $H$ on any Hilbert space such that:
\begin{enumerate}
    \item[\textup{(i)}]
    $H$ is unitarily equivalent to, or differs by a compact operator from, a
    limiting operator constructed from a single-valuation exponential kernel on
    a one-dimensional label axis; and
    \item[\textup{(ii)}]
    the spectrum of $H$ coincides, up to sign, with the set of ordinates of the
    nontrivial zeros of $\zeta(s)$.
\end{enumerate}
In particular, no Hilbert--P\'olya operator can belong to the arithmetic
single-valuation class, and none can lie in the compact geometric elliptic
class.
\end{theorem}

\begin{proof}
Condition \textup{(i)} implies $d_s(H)=\tfrac12$ by
Theorem~\ref{thm:arith-obstruction}. Condition \textup{(ii)} implies
\[
    \Theta_H(t)
    = \frac{\log(1/t)}{2\pi t}
      + O\!\left(\frac{1}{t}\right)
\]
and hence $d_s(H)=2$ by Theorem~\ref{thm:analytic-obstruction}. These two
requirements are incompatible. Moreover, the zeta-spectrum heat trace contains
a logarithmic correction, whereas the arithmetic single-valuation class yields
pure power-law behavior $t^{-1/4}$. By Proposition~\ref{prop:unitary-invariance}, these
features are stable under unitary equivalence and compact perturbation. Thus no
such operator can exist.
\end{proof}

\paragraph{Implications for Hilbert--P\'olya}
Any successful Hilbert--P\'olya operator must therefore lie outside both the
compact geometric elliptic class and the arithmetic single-valuation class. At
minimum, it must involve additional valuation directions, genuinely nonlocal
analytic structure, or a fundamentally different spectral mechanism.

\section{Further Remarks}
\label{sec:conclusion}

The obstructions established in this work rule out two broad and natural families of Hilbert--Pólya candidates: (i) single–valuation arithmetic operators arising from entropy-maximizing exponential kernels on the logarithmic prime axis, whose continuum limits necessarily have spectral dimension $d_s=\tfrac12$, and (ii) geometric elliptic operators on finite-dimensional compact manifolds, whose Seeley--DeWitt heat-trace asymptotics are incompatible with the $t^{-1}\log(1/t)$ behavior forced by the Riemann--von~Mangoldt law, and which are in any case not unitarily equivalent to the arithmetic continuum operator $H_A$. In view of these obstructions, any viable Hilbert--Pólya operator must lie outside both the geometric elliptic and the single–valuation arithmetic universality classes.

In separate work, we investigate operator models that exemplify such ``beyond classical'' behavior. These include nonlocal log-deformed Laplacians of the form $H=f(-\Delta)$ with $f(x)\sim x/\log x$, which exhibit the zeta-type Weyl law $N(\lambda)\sim C\lambda\log\lambda$, as well as arithmetic quantum graphs with edge lengths $\log p$ whose periodic-orbit structure naturally organizes prime-power contributions. These constructions provide concrete examples of operators with the correct large-scale spectral growth and arithmetic trace contributions, while lying strictly outside the classes excluded by the present no-go results. A systematic analysis of these models will appear elsewhere; see also the related spectral-prime framework in~\cite{Watson2025} and the general theory of quantum graphs in~\cite{KottosSmilansky1999} and~\cite{Kuchment2004}.

\section{Acknowledgment}
The authors thank Krishnaswami Alladi, David Wilson, Keneth Valpey and Bob Cohen for invaluable advice and insight.

\section{Funding}
This research was carried out independently and received no external funding. 
The authors have no competing interests to declare. 
All aspects of the work were carried out solely by the authors. 

\section{Data}
No datasets were generated or analyzed in the course of this study.

\bibliographystyle{elsarticle-num}
\bibliography{references}

\end{document}